\documentclass{article}
\usepackage[top=30truemm,bottom=30truemm,left=25truemm,right=25truemm]{geometry}
\usepackage{amsmath,amssymb,amsfonts,mathtools,accents,bm,mathrsfs,bibentry,lscape,multirow,natbib,amsthm,here}
\usepackage{enumerate}
\usepackage{tabu}
\newtheorem{theorem}{Theorem}
\newtheorem{lemma}{Lemma}
\newtheorem{cor}{Corollary}

\newtheorem{remark}{Remark}

\title{On large bandwidth matrix values kernel smoothed estimators for multi-index models}
\author{
Taku MORIYAMA\\
School of Data Science, Yokohama City University}
\date{}

\begin{document}
\maketitle

\begin{abstract}
The kernel smoothing with large bandwidth values causes oversmoothing or underfitting in general. However, when irrelevant variables are included, the corresponding large bandwidth values are known to have an effect of shrinking them. This study investigates asymptotic properties of the kernel conditional density estimator and the regression estimator with large bandwidth matrix elements for cases of multi-index model. It is clarified that the optimal convergence rate of the estimators depends on not the number of the variables but the effective dimension without eliminating the irrelevant variables. Thus, the kernel conditional density estimator and regression estimator are demonstrated to equip the reduction of the curse of dimensionality by nature. Finite sample performances are investigated by a numerical study, and the bandwidth selection is discussed. Finally a case study on the Boston housing data is provided.

\end{abstract}

\section{Introduction}

For a random sequence ${X}_1, \cdots, {X}_n$ the kernel density estimator at the point $x$ is given by $(n h)^{-1} \sum_{i=1}^n k(h^{-1}({x}-{X}_i)) $, where $k$ is the kernel function. $k$ needs to be a probability density function, and $h$ corresponds to the scale parameter. Asymptotic properties of the kernel density estimator has been investigated well, and the consistency for the wide class of distributions is acknowledged. 

\cite{jones1993kernel} surveyed the kernel density estimator with {\it large} bandwidth i.e. $h \to \infty$ as $n \to \infty$. The situation is called oversmoothing or underfitting. It was proven that the kernel density estimator is no longer nonparametric estimator but is possibly consistent estimator for the probability density chosen as the kernel function. It was referred that the variance corrected kernel density estimator converges to the normal distribution if both the underlying distribution and the kernel function is normal. This study extends the study \cite{jones1993kernel} to multivariate cases in kernel density estimation, and considers the application of the results to those in kernel conditional density or regression estimation. 

Multivariate kernel estimators suffer from the curse of dimensionality (i.e. the optimal convergence rate becomes slower as the dimension gets higher), and there exist several studies tackling the issue in regression estimation. If irrelevant variables to the dependent variables exist, by (completely) eliminating the irrelevant variables the estimators converge with the rate depending on the number of relevant variables. There exist some approaches for the kernel regression estimation. An effective algorithm called RODEO is proposed by \cite{lafferty2008rodeo}. \cite{white2017variable} proposed MEKRO that minimizes the measurement error by considering $h = \infty$ under the constraint on the harmonic mean of the bandwidth of the explanatory variables. \cite{feng2018nonparametric} proposed a procedure called nonparametric independence screening and investigated properties in ultrahigh-dimensional case.

\cite{hall2004cross} and \cite{hall2007nonparametric} survey the shrinking property of a cross-validated kernel conditional density estimator and a regression estimator with a diagonal bandwidth matrix, respectively. By taking the bandwidths corresponding to the irrelevant variables diverging to infinity, the optimal convergence rates are attained without eliminating the irrelevant variables. It follows that unlike RODEO or MEKRO kernel conditional density estimator and regression estimator do not need secondary hyperparamter, which is used as a threshold or a constraints, to achieve the optimal convergence rate. That also means nonparametric estimators without eliminating the irrelevant variables completely are irrelevant from variable misspecification (i.e. model misspecification). 

This study first supposes the case that all of or a partial of the explanatory variables is irrelevant in conditional density estimation and regression estimation, where a nonparametric series estimator of the conditional density is proven to attain the optimal minimax convergence rate (\citealt{efromovich2010dimension}). We extend the case conditionally independent to that of partially linear dependent. \cite{xia2008multiple} proved the local-linear regression estimator with a scalar matrix has the optimal convergence rate in the case, proposing its effective calculation algorithm. 

It is proven by \cite{conn2019oracle} that the shrinking property of the kernel regression estimator in case the true is multi-index model (e.g. \citealt{klock2021estimating}). The rates of the asymptotic mean squared error of the kernel estimators with bandwidth possibly diverging are obtained, which shows the shrinking property. It is shown that the shrinking property requires a different condition from that of the case bandwidth $h$ assumed to converge to zero. It is proven that the kernel estimators with any form of positive definite bandwidth matrix attain the optimal convergence rates.

Note that this study investigates asymptotic properties of kernel smoothed estimators with large bandwidth matrix, without structural assumptions on the underlying distribution. Section 2 begins with the basics of this study and investigates asymptotic properties of the multivariate kernel density estimator. This study extends to kernel estimators of conditional density or regression in case all of the explanatory variables are independent of the outcome variable. Section 3 and 4 study the case a partial of explanatory variables are independent and the case the dependency follows a multi-index model, respectively.  One of the important conclusions is that the optimal bandwidth matrix is not diagonal for the multi-index model case. Simulation study and a case study on the Boston housing data set are provided in Section 5. Section 6 concludes this study.

\section{Case of independence in regression or conditional density estimation}

\subsection{Density estimator with large bandwidth matrix}

Let the probability density function of $\bm{X}_1, \cdots, \bm{X}_n$ be $f$. Set $\widehat{f}$ as the multivariate kernel density estimator given by
\begin{align*}
\widehat{f}(\bm{x}) := (n \|\bm{H}\|)^{-1} \sum_{i=1}^n k(\bm{H}^{-1}(\bm{x}-\bm{X}_i)),
\end{align*}
where $k: \mathbb{R}^d \to \mathbb{R}$ is the kernel function and the scale $\bm{H}$ is the $d\times d$ bandwidth matrix, which is regular. $|\bm{H}|$ represents the determinant of the matrix $\bm{H}$ and $\|\bm{H}\|$ is the absolute value. We do not assume $\int k(\bm{z}) {\rm d}\bm{z} = 1$, which is found to be unnecessary later. 

\begin{remark}
If the kernel $k$ is spherically symmetric i.e. there exists $\ell: \mathbb{R}^{d} \to \mathbb{R}$ s.t. $k(\bm{w}) = \ell(\|\bm{w}\|_2)$, then for any $\bm{H}$ and $\bm{w}$ there exists a symmetric matrix $\widetilde{\bm{H}}$ s.t.  $k(\bm{H}^{-1}\bm{w})= k(\widetilde{\bm{H}}^{-1}\bm{w})$, where $\widetilde{\bm{H}}$ is obtained by the polar decomposition.

Product kernels given by $k(\bm{w}) = \prod_{j=1}^d s(w_j^2)$, where $s:\mathbb{R} \to \mathbb{R}$, are spherically symmetric. The univariate Gaussian kernel and the Epanechnikov kernel satisfy the assumption.
\end{remark}

It is known that all the optimal bandwidth matrix elements of the kernel density estimator are of the same order (\citealt{wand1995kernel}), which means there exists a constant matrix $\bm{C}$ s.t. $\bm{H} = h \bm{C}$ i.e. $(\bm{H})_{ij} = O(h)$.

In this study $\bm{H}$ is not necessarily symmetric and $(\bm{H})_{ij}<0$ is allowed (i.e. $(\bm{H})_{ij} \in \mathbb{R}$). 
This study allows elements of the bandwidth matrix $\bm{H}$ being divergent, or more specifically the absolute value  tends to infinity as $n \to \infty$. Since the density estimator $\widehat{f}$ cannot be consistent for the underlying density $f$ unless $\bm{H} \to \bm{O}$, considering the bandwidth matrix $\bm{H}$ s.t. $\bm{H} \nrightarrow \bm{O}$ is usually meaningless in probability density estimation as follows.

\begin{lemma}
Suppose {\rm (i)} $h \to \infty$ as $n \to \infty$, {\rm (ii)} $k(\bm{0})$ exists, {\rm (iii)} $k$ is two times continuously differentiable around the origin, {\rm (iv)} $\nabla (k)(\bm{0}) = \bm{0}$ and {\rm (v)} $2$th moment of $f$ exist. Then, 
$$\|\bm{H}\| \widehat{f}(\bm{x}) - k(\bm{0}) = O(h^{-2}) + O_P(n^{-1/2} h^{-1}).$$
\end{lemma}

\begin{proof}
For a large bandwidth matrix $\bm{H}$ the following expansion holds
\begin{align*}
\mathbb{E}[\widehat{f}(\bm{x})] =& \|\bm{H}\|^{-1} \mathbb{E}\left[ k\left(\bm{H}^{-1} (\bm{x} - \bm{X}_i)\right) \right] = \|\bm{H}\|^{-1}\int k\left(\bm{H}^{-1} (\bm{x} - \bm{w})\right) f(\bm{w}) {\rm d}\bm{w} \\
=& \|\bm{H}\|^{-1} \int [ k(\bm{0}) + \{\bm{H}^{-1} (\bm{x} - \bm{w})\}' \{\nabla (k)(\bm{0})\} + O(h^{-2})] f(\bm{w}) {\rm d}\bm{w} \\
\sim& \|\bm{H}\|^{-1} [k(\bm{0}) + (\bm{x} - \bm{\mu})'  \bm{H}^{-1} \{\nabla (k)(\bm{0})\} + O(h^{-2})];
\end{align*}
hence,
\begin{align*}
\mathbb{E}[\widehat{f}(\bm{x})] - \|\bm{H}\|^{-1} k(\bm{0}) = O(\|\bm{H}\|^{-1} h^{-2}).
\end{align*}

Since
\begin{align*}
\mathbb{E}\left[\|\bm{H}\|^{-2} k^2\left(\bm{H}^{-1} (\bm{x} - \bm{X}_i)\right) \right] = \|\bm{H}\|^{-2} \int [ k^2(\bm{0}) + 2 k(\bm{0}) \{\bm{H}^{-1} (\bm{x} - \bm{w})\}' \{\nabla (k)(\bm{0})\} + O(h^{-2})] f(\bm{w}) {\rm d}\bm{w},
\end{align*}
\begin{align*}
\mathbb{E}[\widehat{f}^2(\bm{x})] - (\mathbb{E}[\widehat{f}(\bm{x})])^2 
=& \frac{1}{n} \|\bm{H}\|^{-2} \int [ k^2(\bm{0}) + 2 k(\bm{0}) \{\bm{H}^{-1} (\bm{x} - \bm{w})\}' \{\nabla (k)(\bm{0})\} + O(h^{-2})] f(\bm{w}) {\rm d}\bm{w} \\
&~~~ - \frac{1}{n} \|\bm{H}\|^{-2} \{\int [ k(\bm{0}) +  \{\bm{H}^{-1} (\bm{x} - \bm{w})\}'  \{\nabla (k)(\bm{0})\} + O(h^{-2})] f(\bm{w}) {\rm d}\bm{w}\}^2 \\
\sim& \frac{1}{n} \|\bm{H}\|^{-2}  \{\nabla (k)(\bm{0})\} \bm{H}^{-1} \mathbb{V}[\bm{X}] \bm{H}^{-1} \{\nabla (k)(\bm{0})\}'.
\end{align*}
The lemma has been proved.
\end{proof}

\begin{remark}
Comparing the assumption with a usual small bandwidth matrix with that of Lemma 1 we see that assumptions are made on the kernel function rather than the underlying distribution. However, it is not restrictive and the usual kernel e.g. the Gaussian or the Epanechnikov type satisfies the assumption. The assumption on the moment relates the tail-heavyness of the underlying distribution.
\end{remark}

\begin{remark}
Under the assumption of the kernel $k$ being a probability density function, the expectation of the kernel density estimator with the large bandwidth matrix given by
\begin{align*}
\int \|\bm{H}\|^{-1} k\left(\bm{H}^{-1} (\bm{x} - \bm{w})\right) f(\bm{w}) {\rm d}\bm{w}
\end{align*}
can be seen as the convolution. Then, this is also a probability density function, which corresponds to the random variable $\bm{X} + \bm{H} \bm{K}$. $\bm{K}$ is the random variable independent of $\bm{X}$, whose probability density function is $k$. Due to the scaling (magnifying) $\bm{H}$ for large enough $n$ the random variable is dominated by $\bm{H} \bm{K}$, which is asymptotically uniform on $\mathbb{R}^{d}$
\end{remark}

Before proceeding to the kernel estimator of conditional density or regression, we note the following points.

\begin{remark}
Besides the kernel density estimator, the kernel estimator of conditional density or regression introduced later follow asymptotically normal, which is proven by the Slutsky's Theorem.
\end{remark}

\begin{remark}
Lemma 1 means $\widehat{f}(\bm{x}) \overset{p}{\nrightarrow} {f}(\bm{x})$ when ${f}(\bm{x})\neq k(\bm{0})$. However, as long as ${f}(\bm{x}) -k(\bm{0})=o(1)$ the consistency still holds. Examples include the contamination in distribution as a case, which is important in the following discussion of conditional density or regression estimation.
\end{remark}

\begin{remark}
In the sense of pointwise error such as the mean squared error (i.e. pointwise estimation) the model does not have to follows the limiting one globally. That means the consistency requires in this case not ${f}(\bm{x}) \equiv k(\bm{0})$ but only ${f}(\bm{x}) = k(\bm{0})$.
\end{remark}

Based on the lemma we discusses the kernel estimator of conditional density or regression with large bandwidth matrix  next.

\subsection{Case of independence in regression estimation}

Let the $d$-dimensional i.i.d. r.v. be composed of $\bm{X}_i := (\bm{X}_{1,i}', \bm{X}_{2,i}')'$ and the probability density function of the $d_2$-dimensional random variable $\bm{X}_{2,i}$ be $f_2$. Set $\widehat{f}_2$ is the $d_2$-dimensional kernel density estimator given by
\begin{align*}
\widehat{f}_2(\bm{x}_2) := (n \|\bm{H}_{22}\|)^{-1} \sum_{i=1}^n k_{2}(\bm{H}_{22}^{-1}(\bm{x}_2-\bm{X}_{2,i})),
\end{align*}
where $k_{2}: \mathbb{R}^{d_2} \to \mathbb{R}$ is the kernel function, $\bm{H}_{22}$ is the $d_2 \times d_2$ bandwidth matrix, $(\bm{H}_{22})_{ij} \in \mathbb{R}$ and $(\bm{H}_{22})_{ij} = O(h_{22})$. The kernel regression estimator is defined as
\begin{align*}
\widehat{m}(\bm{x}_2) := \widehat{f}_2(\bm{x}_2)^{-1} (n \|\bm{H}_{22}\|)^{-1} \sum_{i=1}^n \bm{X}_{1,i} k_{2}(\bm{H}_{22}^{-1}(\bm{x}_2-\bm{X}_{2,i})).
\end{align*}
Then, the following result on the kernel regression estimator with large bandwidth matrix is obtained.

\begin{theorem}
Suppose {\rm (i)} $h_{22} \to \infty$ as $n \to \infty$, {\rm (ii)} $k_2(\bm{0}) \neq 0$ exists {\rm (iii)} $k_2$ is two times continuously differentiable around the origin, {\rm (iv)} $\nabla (k_2)(\bm{0}) = \bm{0}$ and {\rm (v)} $2$nd moments of $f_1$ and $f_2$ exist. Then, 
$$\mathbb{E}[\{\widehat{m}(\bm{x}_2) - \mathbb{E}[\bm{X}_{1,i}]\}^2] = O(h_2^{-4} + n^{-1}).$$
\end{theorem}

\begin{proof}
The expectation of $(n \|\bm{H}_{22}\|)^{-1} \sum_{i=1}^n \bm{X}_{1,i} k_{2}(\bm{H}_{22}^{-1}(\bm{x}_2-\bm{X}_{2,i}))$ is given by
\begin{align*}
\|\bm{H}_{22}\|^{-1}\int \bm{w}_1 k_2\left(\bm{H}_{22}^{-1} (\bm{x}_2 - \bm{w}_2)\right) f(\bm{w}) {\rm d}\bm{w} =&  \|\bm{H}_{22}\|^{-1} \int \bm{w}_1 [ k_2(\bm{0}) + O(h_{22}^{-2})] f(\bm{w}) {\rm d}\bm{w} \\
\sim& \|\bm{H}_{22}\|^{-1} \mathbb{E}[\bm{X}_{1,i}] [k_2(\bm{0}) + O(h_{22}^{-2})].
\end{align*}

The variance is 
\begin{align*}
& \frac{1}{n} \|\bm{H}_{22}\|^{-2} \int \bm{w}_1 \bm{w}_1' [ k_2^2(\bm{0}) + O(h_{22}^{-2})] f(\bm{w}) {\rm d}\bm{w} \\
&~~~ -  \frac{1}{n} \|\bm{H}_{22}\|^{-2} \{\int \bm{w}_1 [ k_2(\bm{0}) + O(h_{22}^{-2})] f(\bm{w}) {\rm d}\bm{w}\} \{\int \bm{w}_1 [ k_2(\bm{0}) + O(h_{22}^{-2})] f(\bm{w}) {\rm d}\bm{w}\}' \\
\sim& \frac{1}{n} \|\bm{H}_{22}\|^{-2}  k_2^2(\bm{0}) \mathbb{V}[\bm{X}_{1,i}].
\end{align*}
Then, we have 
\begin{align*}
\mathbb{E}[\{  \|\bm{H}_{22}\|^{-1} \mathbb{E}[\bm{X}_{1,i}] k_2(\bm{0}) - 
(n \|\bm{H}_{22}\|)^{-1} \sum_{i=1}^n \bm{X}_{1,i} k_{2}(\bm{H}_{22}^{-1}(\bm{x}_2-\bm{X}_{2,i}))
\}^2] = O(\|\bm{H}_{22}\|^{-2} (h_{22}^{-4} + n^{-1})).
\end{align*}
Applying the H\"{o}lder's inequality to the AMSE and the result of Lemma 1 on $\widehat{f}_2(\bm{x}_2)$ we have the result.
\end{proof}

\begin{cor}
Suppose the conditions of Theorem 1. If $h_{22}^{-4} = O(n^{-1})$ or $o(n^{-1})$, the MSE with the optimal bandwidth is of order $n^{-1}$. 
\end{cor}

Corollary 1 shows the kernel regression estimator converges to the unconditional expectation $\mathbb{E}[\bm{X}_{1,i}]$, where the rate of the convergence is $\sqrt{n}$. The kernel estimator is known as the local constant estimator, and has already been known to be unbiased in the case under the usual bandwidth condition ($h_{22} \to 0$), which results in $\sqrt{n}$-consistency attained by a large enough bandwidth matrix (to be precise $h_{22} = O(1)$ imaginarily).

As seen from the proof, the kernel function and the bandwidth matrix of the numerator and the denominator $\widehat{f}_2$ must be same, different from that of small (i.e. usual) bandwidth matrix.

\subsection{Case of independence in conditional density estimation}

Set $\widehat{f}$ is the kernel density function of $\bm{X}_i$ with the kernel function $k: \mathbb{R}^d \to \mathbb{R}$ and the bandwidth matrix
\begin{align*}
 \bm{H} = \left(\begin{matrix}
 \bm{H}_{11} & \bm{H}_{12} \\
 \bm{H}_{21} & \bm{H}_{22} \\
 \end{matrix}
 \right),
\end{align*}
where $\bm{H}_{k\ell}$ is the $d_k \times d_{\ell}$ matrix and $(\bm{H}_{k\ell})_{ij} \in \mathbb{R}$. All the elements of every blocks are supposed to be same order, or more specifically, $(\bm{H}_{k\ell})_{ij} = O(h_{k\ell})$. Then, we have the following result.

\begin{lemma}
Suppose {\rm (i)} $\bm{H}_{11}$ is regular, $h_{11} \to 0$ and $h_{22} \to \infty$ as $n \to \infty$, {\rm (ii)}
$$\bm{H}_{12}=\bm{O} ~~~ or ~~~ h_{12} = O(h_{11})$$
and
$$\bm{H}_{21}= \bm{O} ~~~ or ~~~ h_{21} = O(h_{11}),$$
{\rm (iii)} $k_2(\bm{0})$ exists, where
$$k_2(\bm{z}_2) = \int k(\bm{z}_1, \bm{z}_2) {\rm d}\bm{z}_1,$$
{\rm (iv)} $k$ is two times continuously differentiable on $\{(\bm{z}_1, \bm{0}) | \bm{z}_1 \in \mathbb{R}^{d_1}\}$, {\rm (v)} $\nabla (k)(\bm{z}_1, \bm{0}) = \bm{0}$ for any $\bm{z}_1$, {\rm (vi)} $\int \bm{z}_1 k(\bm{z}_1, \bm{0}) {\rm d}\bm{z}_1 = \bm{0}$, {\rm (vii)} $f$ is partially continuously differentiable with respect to the variable $\bm{x}_1$ for any $\bm{x}_2$. Then, 
\begin{align*}
\|\bm{H}_{22}\| \widehat{f}(\bm{x}) - k_2(\bm{0}) {f}_1(\bm{x}_1) = O_P(n^{-1/2} \|\bm{H}_{11}\|^{-1/2}) + O(h_{11}^2 + h_{22}^{-2}).
\end{align*}
where $f_1$ is the probability density function of $\bm{X}_{1,i}$.
\end{lemma}

\begin{proof}
Suppose $\bm{H}_{12}\neq\bm{O}$ and $\bm{H}_{21}\neq\bm{O}$ in this proof. The regularity of $\bm{H}$ does not mean $\bm{H}_{11}$ is invertible; however, let $\bm{\Xi}^{-1} := \bm{H}_{11}^{-1} + \bm{H}_{11}^{-1} \bm{H}_{12} \bm{\Upsilon}^{-1} \bm{H}_{21} \bm{H}_{11}^{-1}$ and $\bm{\Upsilon} := \bm{H}_{22} - \bm{H}_{21} \bm{H}_{11}^{-1} \bm{H}_{12}$ for notation. It holds that 
$$\bm{H}^{-1} =\left(\begin{matrix}
\bm{\Xi}^{-1} & -\bm{H}_{11}^{-1} \bm{H}_{12} \bm{\Upsilon}^{-1}\\
-\bm{\Upsilon}^{-1} \bm{H}_{21} \bm{H}_{11}^{-1} & \bm{\Upsilon}^{-1} \\
 \end{matrix}
 \right).$$
On the determinant of the block matrix it holds
\begin{align*}
|\bm{H}|^{-1} = |\bm{H}_{11}|^{-1} |\bm{H}_{22} - \bm{H}_{21} \bm{H}_{11}^{-1} \bm{H}_{12}|^{-1} = |\bm{H}_{11}|^{-1} |\bm{\Upsilon}|^{-1}
\end{align*}
and the Sylvester's determinant theorem yields
\begin{align*}
|\bm{\Xi}|^{-1} =& |\bm{\Xi}^{-1}| = |\bm{H}_{11}|^{-2} |\bm{H}_{11} + \bm{H}_{12} \bm{\Upsilon}^{-1} \bm{H}_{21}| \\
=& |\bm{H}_{11}|^{-1} |\bm{\Upsilon}|^{-1} |\bm{\Upsilon} + \bm{H}_{21} \bm{H}_{11}^{-1} \bm{H}_{12}| = |\bm{H}_{11}|^{-1} |\bm{\Upsilon}|^{-1} |\bm{H}_{22}|.
\end{align*}
Combining them, we have $\|\bm{H}\| \|\bm{\Xi}^{-1}\| = \|\bm{H}_{22}\|$. 

$(\bm{\Xi}^{-1})_{ij} = O(h_{11}^{-1})$, $(\bm{\Upsilon}^{-1})_{ij} = O(h_{22}^{-1})$, $(-\bm{H}_{11}^{-1} \bm{H}_{12} \bm{\Upsilon}^{-1})_{ij} = O(h_{22}^{-1})$, $(-\bm{\Upsilon}^{-1} \bm{H}_{21} \bm{H}_{11}^{-1} \bm{\Xi})_{ij} = O(h_{11} h_{22}^{-1})$ i.e. the order of the elements of $\bm{H}^{-1}$ is the form of
$$\left(\begin{matrix}
h_{11}^{-1} & h_{22}^{-1} \\
h_{22}^{-1} & h_{22}^{-1} \\
 \end{matrix}
 \right).$$
By changing variable $\bm{\Xi}^{-1} (\bm{x}_1 - \bm{w}_1) = \bm{z}_1$ and $(\bm{x}_2 - \bm{w}_2) = \bm{z}_2$, $\mathbb{E}[\widehat{f}(\bm{x})] = \|\bm{H}\|^{-1} \int k\left(\bm{H}^{-1} (\bm{x} - \bm{w})\right) f(\bm{w}) {\rm d}\bm{w}$ is given by
\begin{align*}
\mathbb{E}[\widehat{f}(\bm{x})] =& \|\bm{H}\|^{-1} \|\bm{\Xi}\| \int [k(\bm{z}_1, \bm{0}) + \{(-\bm{z}_2' (\bm{\Upsilon}^{-1})' \bm{H}_{12}' (\bm{H}_{11}^{-1})', \\
& ~~~ -\bm{z}_1' \bm{\Xi}' (\bm{H}_{11}^{-1})' \bm{H}_{21}' (\bm{\Upsilon}^{-1})' + \bm{z}_2'(\bm{\Upsilon}^{-1})')' \nabla(k) (\bm{z}_1, \bm{0})\} + O(h_{22}^{-2})] \\
& \times \{f(\bm{x}_1, \bm{x}_2 -\bm{z}_2) - \bm{z}_1' \bm{\Xi}' \nabla_1(f)(\bm{x}_1, \bm{x}_2 - \bm{z}_2) + O(h_{11}^2) \} {\rm d}\bm{z}_1 {\rm d}\bm{z}_2 \\
=& \|\bm{H}_{22}\|^{-1}| \{k_2(\bm{0}) f_1(\bm{x}_1) + O(h_{11}^2 + h_{22}^{-2})\}, \\
\mathbb{E}[\|\bm{H}\|^{-2} & k^2\left(\bm{H}^{-1} (\bm{x} - \bm{X}_i)\right)] \sim \|\bm{H}\|^{-2} \|\bm{\Xi}\| \int k^2(\bm{z}_1, \bm{0}) f(\bm{x}_1, \bm{x}_2 -\bm{w}_2) {\rm d}\bm{z}_1 {\rm d}\bm{w}_2
\end{align*}
where $\nabla_1$ is the vector differential operator applied to the first $d_1$ variables. Thus, we have
\begin{align*}
\mathbb{E}[\widehat{f}^2(\bm{x})] - (\mathbb{E}[\widehat{f}(\bm{x})])^2 
\sim& \frac{1}{n} \|\bm{H}\|^{-1} \|\bm{H}_{22}\|^{-1} \left[ f_1(\bm{x}_1) \int k^2(\bm{z}_1, \bm{0}) {\rm d}\bm{z}_1 - \|\bm{\Xi}\| \{ k_2(\bm{0}) f_1(\bm{x}_1)\}^2 \right].
\end{align*}

For the rest case $\bm{H}_{12}=\bm{O}$ or $\bm{H}_{21}=\bm{O}$, still $(\bm{\Xi}^{-1})_{ij} = O(h_{11}^{-1})$ and $(\bm{\Upsilon}^{-1})_{ij} = O(h_{22}^{-1})$, and the above proof follows in the same way.
\end{proof}

Lemma 2 shows the kernel conditional density estimator converges to the unconditional density $f_1$, where the rate of the convergence can depend on not $d=d_1 + d_2$ but $d_1$. 

The Slutsky's theorem states the conditional density estimator $(\widehat{f}_2 (\bm{x}_2))^{-1} \widehat{f}(\bm{x})$ stochastically converges to $f_1(\bm{x}_1)$ if both $\widehat{f}(\bm{x})$ and $\widehat{f}_2(\bm{x}_2)$ are consistent if the corresponding bandwidth matrix $\bm{H}_{22}$ is shared and the kernel function of $\widehat{f}$ satisfies $\int k(\bm{z}_1, \bm{0}) {\rm d}\bm{z}_1 = k_2(\bm{0})$, where $k_2$ is the kernel function of $\widehat{f}_2$.

Combining Lemma 1 and 2 we have the following consequence.

\begin{theorem}
Suppose the conditions of Lemma 2 regarding $\widehat{f}(\bm{x})$. Then, 
\begin{align*}
\mathbb{E}[ \{ \widehat{f}_2(\bm{x}_2)^{-1} \widehat{f}(\bm{x}) - {f}_1(\bm{x}_1) \}^2] = O(h_{11}^4 + h_{22}^{-4} + n^{-1} \|\bm{H}_{11}\|^{-1}).
\end{align*}
\end{theorem}

\begin{proof}
Under the conditions of Lemma 2, the assumption of Lemma 1 regarding $\widehat{f}_2(\bm{x}_2)$ holds. The AMSE is asymptotically equals to
$$\{k_2(\bm{0})\}^{-2} \mathbb{E}[ \{
\|\bm{H}_{22}\| (\widehat{f}(\bm{x})- \|\bm{H}_{22}\|^{-1} k_2(\bm{0}) f_1(\bm{x}_1)) - \|\bm{H}_{22}\|^2 f_1(\bm{x}_1) (\widehat{f}_2(\bm{x}_2) -k_2(\bm{0}))
\}^2].$$
Applying Lemma 1 and 2 and the H\"{o}lder's inequality we have the desired result.
\end{proof}

Under the independent condition the conditional density estimation is merely the marginal density estimation. The following result show the optimal convergence rate coincides with that of the kernel (marginal) density estimator.
\begin{cor}
Suppose the conditions of Theorem 2. If $h_{22}^{-1}=O(h_{11})$ or $o(h_{11})$, minimizing the AMSE yields $h_{11}=O(n^{-(d_1 +4)^{-1}})$ and the MSE with the optimal bandwidth is of order $n^{-4(d_1 +4)^{-1}}$. 
\end{cor}

\begin{remark}
Under the assumption of the kernel $k$ being a probability density function, the expectation of the kernel density estimator with the blockwise bandwidth matrix can be seen as the probability density function of the random variable $\bm{X} + \bm{H} \bm{K}$. $\bm{K} := (\bm{K}_{1}', \bm{K}_{2}')'$ is the random variable independent of $\bm{X}$, whose probability density function is $k$. Due to the magnifying and shrinking $\bm{H}$ for large enough $n$ the random variable is dominated by $(\bm{X}_{1}', (\bm{H}_{22} \bm{K}_2)')'$, where the last $d_2$ variable asymptotically follows uniform on $\mathbb{R}^{d_2}$.
\end{remark}

\begin{remark}
Let
\begin{align*}
\bm{H} = \left(\begin{matrix}
 h_{11} \bm{C}_{11} & h_{11} \bm{C}_{12} \\
 h_{11} \bm{C}_{21} & h_{11} \bm{C}_{22} \\
 \end{matrix}
 \right),
\end{align*}
where $\bm{C}_{11}$, $\bm{C}_{12}$, $\bm{C}_{21}$ and $\bm{C}_{22}$ are constant matrices with the size of $d_1 \times d_1$, $d_1 \times d_2$, $d_2 \times d_1$, and $d_2 \times d_2$ satisfying $\bm{C}_{11} \neq \bm{O}$ being regular and $\bm{C}_{22} \neq \bm{O}$. 
\begin{align*}
\bm{E}_n := \left(\begin{matrix}
 \bm{E}_{11} & \bm{O} \\
 \bm{O} & h_{11}^{-1} h_{22} \bm{E}_{22} \\
 \end{matrix}
 \right),
\end{align*}
where $\bm{E}_{11}$ and $\bm{E}_{22}$ is the identity matrices with the size $d_1 \times d_1$ and $d_2 \times d_2$.  Then, 
\begin{align*}
\bm{E}_n \bm{H} = \left(\begin{matrix}
h_{11} \bm{C}_{11} & h_{11} \bm{C}_{12} \\
h_{22} \bm{C}_{21} & h_{22} \bm{C}_{22} \\
 \end{matrix}
 \right),
\end{align*}
and $\bm{E}_n \bm{H}$ satisfies the assumption on the bandwidth matrix (i.e. {\rm (i)}--{\rm (ii)} in Lemma 2) if $\bm{C}_{21} = \bm{O}$.

Since
\begin{align*}
& (n \|\bm{E}_n \bm{H}\|)^{-1} \sum_{i=1}^n k((\bm{E}_n \bm{H})^{-1} (\bm{x}-\bm{X}_i))\\
=& \|\bm{E}_n\|^{-1} (n \|\bm{H}\|)^{-1} \sum_{i=1}^n k(\bm{H}^{-1} ((\bm{x}_1-\bm{X}_{1,i})', h_{22}^{-1} h_{11}(\bm{x}_2-\bm{X}_{2,i})')'),
\end{align*}
$\widehat{f}$ with the (partially) large bandwidth matrix $\bm{E}_n \bm{H}$ is $\|\bm{E}_n\|^{-1}$ times the the kernel density estimator with the small bandwidth matrix $\bm{H}$ for the random variable $(\bm{X}_{1,i}', h_{22}^{-1} h_{11} \bm{X}_{2,i}')'$.

The expectation of the kernel density estimator can be seen to be the probability density function of the random variable 
$$\bm{X}_i+ \bm{E}_n \bm{H} \bm{K}$$
and $\|\bm{E}_n\|^{-1}$ times that of
$$(\bm{X}_{1,i}', h_{22}^{-1} h_{11} \bm{X}_{2,i}')' + \bm{H} \bm{K},$$
where $\bm{K}$ is the random variable independent of $(\bm{X}_{1,i}', h_{22}^{-1} h_{11} \bm{X}_{2,i}')'$ for every $n$. The random variable is dominated by $(\bm{X}_{1}', \bm{0}')'$, where the last $d_3$ variable asymptotically shrinks.
\end{remark}

\begin{remark}
The bandwidth matrix assumption 
$$\bm{H}_{12}=\bm{O} ~~~ or ~~~ h_{12} = O(h_{11})$$
and
$$\bm{H}_{21}= \bm{O} ~~~ or ~~~ h_{21} = O(h_{11})$$
with $h_{11} \to 0$ and $h_{22} \to \infty$ as $n \to \infty$, is sufficient to make
$$\widehat{f}_2(\bm{x}_2)^{-1} \widehat{f}(\bm{x}) \overset{p}{\to} {f}_1(\bm{x}_1).$$
For $h_{21} \to \infty$ the bandwidth of $\widehat{f}_2$ needs to be well-designed (see Appendix: Additional results).
\end{remark}

\section{Case of conditional independence in regression or conditional density estimation}

This section proves the kernel estimators with all possible bandwidth matrices attains the optimal convergence rates. This suggests bandwidth selection methods asymptotically minimizing the mean integrated squared errors choose the theoretically optimal ones without specifying the dependence structure. The next section studies the numerical properties of bandwidth selection methods.

\subsection{Case of conditional independence in regression estimation}

Let the $d$-dimensional r.v. be $\bm{X}_i := (\bm{X}_{1,i}', \bm{X}_{2,i}', \bm{X}_{3,i}')'$, $\bm{Z}_i = (\bm{X}_{2,i}', \bm{X}_{3,i}')'$ and the probability density function of the $d_2 + d_3$-dimensional random variable $\bm{Z}_i$ be $f_{23}$. Set $\widehat{f}_{23}$ is the kernel density estimator of $\bm{Z}_i$ given by
\begin{align*}
\widehat{f}_{23}(\bm{z}) := (n \|\bm{\Omega}\|)^{-1} \sum_{i=1}^n k_{23}(\bm{\Omega}^{-1}(\bm{z}-\bm{Z}_i)),
\end{align*}
$k_{23} : \mathbb{R}^{d_2 + d_3} \to \mathbb{R}$ and the bandwidth is supposed to be the $(d_2 +d_3) \times (d_2 + d_3)$ matrix 
\begin{align*}
\bm{\Omega} = \left(\begin{matrix}
 \bm{H}_{22} & \bm{H}_{23} \\
 \bm{H}_{32} & \bm{H}_{33} \\
 \end{matrix}
 \right),
\end{align*}
where all the elements of every blocks are supposed to be same order. Let
$$\bm{\Omega}^{-1} :=\left(\begin{matrix}
\bm{\Xi}^{-1} & -\bm{H}_{22}^{-1} \bm{H}_{23} \bm{\Upsilon}^{-1}\\
-\bm{\Upsilon}^{-1} \bm{H}_{32} \bm{H}_{22}^{-1} & \bm{\Upsilon}^{-1} \\
 \end{matrix}
 \right),$$
where $\bm{\Upsilon} = \bm{H}_{33} - \bm{H}_{32} \bm{H}_{22}^{-1} \bm{H}_{23}$ and $\bm{\Xi}^{-1} = \bm{H}_{22}^{-1} + \bm{H}_{22}^{-1} \bm{H}_{23} \bm{\Upsilon}^{-1} \bm{H}_{32} \bm{H}_{22}^{-1}$. The following result on the kernel regression estimator 
\begin{align*}
\widehat{m}(\bm{z}) := \widehat{f}_{23}(\bm{z})^{-1} (n \|\bm{\Omega}\|)^{-1} \sum_{i=1}^n \bm{X}_{1,i} k_{23}(\bm{\Omega}^{-1}(\bm{z}-\bm{Z}_i)),
\end{align*}
is also the consequence of Lemma 2.

\begin{theorem}
Suppose {\rm (i)} $\bm{H}_{22}$ is regular, $h_{22} \to 0$ and $h_{33} \to \infty$, {\rm (ii)} 
$$\bm{H}_{23}=\bm{O} ~~~ and  ~~~ \bm{H}_{32}= \bm{O}$$
or
$$h_{23} = O(h_{22}) ~~~ and ~~~ h_{32} = O(h_{22}),$$
{\rm (iii)} $k_3(\bm{0}) \neq 0$ exists, where
$$k_{3}(\bm{z}_3) = \int k_{23}(\bm{z}_2, \bm{z}_3) {\rm d}\bm{w}_2,$$
{\rm (iv)} $k_{23}$ is two times continuously differentiable on $\{(\bm{z}_2, \bm{0}) | \bm{z}_2 \in \mathbb{R}^{d_2}\}$, {\rm (v)} $\int \bm{z}_2 k_{23}(\bm{z}_2, \bm{0}) {\rm d}\bm{z}_2 = \bm{0}$, {\rm (vi)} $1$st moment of $f$ exist, {\rm (vii)} $f$ is partially continuously differentiable with respect to the variable $\bm{x}_2$ for any $\bm{x}_3$. Then, it holds that
\begin{align*}
\mathbb{E}[\{\widehat{m}(\bm{z}) - m(\bm{x}_2)\}^2] = O(h_{22}^4 + h_{33}^{-4} + n^{-1} \|\bm{H}_{22}\|^{-1}),
\end{align*}
where
$$m(\bm{x}_2) := \int \bm{w}_{1} \frac{f_{12}(\bm{w}_1, \bm{x}_2)}{f_{2}(\bm{x}_2)} {\rm d}\bm{w}_1,$$
$f_{12}$ and $f_2$ are the probability density function of $(\bm{X}_{1,i},\bm{X}_{2,i})$ and $\bm{X}_{2,i}$ respectively.
\end{theorem}

\begin{proof}
By changing variable $\bm{\Xi}^{-1} (\bm{x}_2 - \bm{w}_2) = \bm{z}_2$ similarly as the proof of Lemma 2, 
\begin{align*}
\|\bm{\Omega}\|^{-1} \mathbb{E}[\bm{X}_{1,i} k_{23}(\bm{\Omega}^{-1}(\bm{z} -\bm{Z}_i))] =& \|\bm{\Omega}\|^{-1} \|\bm{\Xi}\| \int \bm{w}_{1} [k_{23}(\bm{z}_2, \bm{0}) + \{(-\bm{w}_3' (\bm{\Upsilon}^{-1})' \bm{H}_{23}' (\bm{H}_{22}^{-1})', \\
& ~~~ -\bm{z}_2' \bm{\Xi}' (\bm{H}_{22}^{-1})' \bm{H}_{32}' (\bm{\Upsilon}^{-1})' + \bm{w}_3'(\bm{\Upsilon}^{-1})')' \nabla(k_{23}) (\bm{z}_2, \bm{0})\} + O(h_{33}^{-2})\}] \\
& \times \{f(\bm{w}_1, \bm{x}_2, \bm{w}_3) - \bm{z}_2' \bm{\Xi}' \nabla_2(f)(\bm{w}_1, \bm{x}_2, \bm{w}_3) + O(h_{22}^2) \} {\rm d}\bm{w}_1 {\rm d}\bm{z}_2 {\rm d}\bm{w}_3 \\
=& \|\bm{H}_{33}\|^{-1} \{k_3(\bm{0}) \int \bm{w}_{1} f(\bm{w}_1, \bm{x}_2, \bm{w}_3) {\rm d}\bm{w}_1 {\rm d}\bm{w}_3 + O(h_{22}^2 + h_{33}^{-2})\}.
\end{align*}
Similarly, we have
\begin{align*}
\|\bm{\Omega}\|^{-2} \mathbb{E}[\bm{X}_{1,i} \bm{X}_{1,i}' k_{23}^2(\bm{\Omega}^{-1}(\bm{z}-\bm{Z}_i))] \sim |\bm{\Omega}\|^{-2} \|\bm{\Xi}\| \int \bm{w}_1 \bm{w}_1' k_{23}^2(\bm{z}_2, \bm{0}) f(\bm{w}_1, \bm{x}_2, \bm{w}_3) {\rm d}\bm{w}_1 {\rm d}\bm{z}_2 {\rm d}\bm{w}_3
\end{align*}
and the asymptotic variance of $(n \|\bm{\Omega}\|)^{-1} \sum_{i=1}^n \bm{X}_{1,i} k_{23}(\bm{\Omega}^{-1}(\bm{z} -\bm{Z}_i))$ is
\begin{align*}
& \frac{1}{n} |\bm{\Omega}\|^{-2} \|\bm{\Xi}\| \biggm[\int \bm{w}_1 \bm{w}_1' k_{23}^2(\bm{z}_2, \bm{0}) f(\bm{w}_1, \bm{x}_2, \bm{w}_3) {\rm d}\bm{w}_1 {\rm d}\bm{z}_2 {\rm d}\bm{w}_3 \\
& ~~~ - \|\bm{\Xi}\| k_3^2(\bm{0}) \{\int \bm{w}_1 f(\bm{w}_1, \bm{x}_2, \bm{w}_3) {\rm d}\bm{w}_1 {\rm d}\bm{w}_3\} \{\int \bm{w}_1 f(\bm{w}_1, \bm{x}_2, \bm{w}_3) {\rm d}\bm{w}_1 {\rm d}\bm{w}_3\}' \biggm]. 
\end{align*}
Therefore, the order of the AMSE is obtained as seen from applying Lemma 2 on $\widehat{f}_{23}(\bm{z})$ and the H\"{o}lder's inequality.
\end{proof}

Theorem 3 shows the kernel regression estimator converges to the partially conditional expectation $m(\bm{x}_2)$, where the rate of the convergence can depend on not $d_2 + d_3$ but $d_2$. 

\begin{cor}
Suppose the conditions of Theorem 3. If $h_{33}^{-1}=O(h_{22})$ or $o(h_{22})$, minimizing the AMSE yields $h_{22}=O(n^{-(d_2 +4)^{-1}})$ and the MSE with the optimal bandwidth is of order $n^{-4(d_2 +4)^{-1}}$. 
\end{cor}

\subsection{Case of conditional independence in conditional density estimation}

$\widehat{f}$ is the kernel density function of $\bm{X}_i$ with the kernel function $k: \mathbb{R}^d \to \mathbb{R}$ and the bandwidth matrix is supposed to be
\begin{align*}
\bm{H} = \left(\begin{matrix}
 \bm{H}_{11} & \bm{H}_{12} & \bm{H}_{13} \\
 \bm{H}_{21} & \bm{H}_{22} & \bm{H}_{23} \\
 \bm{H}_{31} & \bm{H}_{32}  & \bm{H}_{33} \\
 \end{matrix}
 \right) =: 
 \left(\begin{matrix}
 \bm{\Psi} & \bm{H}_{12,3} \\
 \bm{H}_{3,12} & \bm{H}_{33} \\
 \end{matrix}
 \right),
\end{align*}
where all the elements of every blocks are supposed to be same order.

Let us consider the following case that $\bm{X}_{1,i}$ and $\bm{X}_{3,i} $ are conditionally independent given $\bm{X}_{2,i} = \bm{x}_2$ 
$$\mathbb{P}[\bm{X}_{1,i} = \bm{x}_1, \bm{Z}_i =\bm{z}] = \mathbb{P}[\bm{X}_{1,i} = \bm{x}_1 | \bm{X}_{2,i} = \bm{x}_2] \mathbb{P}[\bm{Z}_{i} = \bm{z}].$$
i.e. $\mathbb{P}[\bm{X}_{1,i} = \bm{x}_1 | \bm{Z}_i =\bm{z}] = \mathbb{P}[\bm{X}_{1,i} = \bm{x}_1 | \bm{X}_{2,i} = \bm{x}_2]$.

Suppose the condition of Lemma 2 regarding both $\widehat{f}(\bm{x})$ and $\widehat{f}_{23}(\bm{z})$ hold, where the bandwidth matrix of $\widehat{f}_{23}(\bm{z})$ is
\begin{align*}
\bm{\Omega} := \left(\begin{matrix}
 \bm{H}_{22} & \bm{H}_{23} \\
 \bm{H}_{32} & \bm{H}_{33} \\
 \end{matrix}
 \right).
\end{align*}
Then, we have the following consequence on the kernel conditional density estimator $(\widehat{f}_{23}(\bm{z}))^{-1} \widehat{f}(\bm{x})$. 

\begin{theorem}
Suppose the conditions of Lemma 2 regarding both $\widehat{f}(\bm{x})$ and $\widehat{f}_{23}(\bm{z})$ i.e. {\rm (i)} $\bm{\Psi}$ is regular, $(\bm{\Psi})_{ij}=O(h_{11})$, $h_{11} \to 0$ and $h_{33} \to \infty$ as $n \to \infty$, {\rm (ii)} 
$$\bm{H}_{12,3}=\bm{O} ~~~ and ~~~ \bm{H}_{3,12}= \bm{O}$$
or 
$$h_{12,3} = O(h_{22}) ~~~ and ~~~ h_{3,12} = O(h_{22}),$$
{\rm (iii)} $k_3(\bm{0}) \neq 0$ exists, where
$$k_{3}(\bm{z}_3) = \int k(\bm{w}_1, \bm{w}_2, \bm{z}) {\rm d}\bm{w}_1 {\rm d}\bm{w}_2,$$
{\rm (iv)} $k$ is two times continuously differentiable on $\{(\bm{z}_1, \bm{z}_2, \bm{0}) | (\bm{z}_1, \bm{z}_2) \in \mathbb{R}^{d_1 + d_2}\}$, {\rm (v)} $\nabla_{3}(k)(\bm{z}_1, \bm{z}_2, \bm{0}) = \bm{0}$ for any $(\bm{z}_1, \bm{z}_2)$, {\rm (vi)} $\int (\bm{z}_1', \bm{z}_2')' k(\bm{z}_1, \bm{z}_2, \bm{0}) {\rm d}\bm{z}_1 {\rm d}\bm{z}_2 = \bm{0}$, {\rm (vii)} $f$ is partially continuously differentiable with respect to the variable $(\bm{x}_1, \bm{x}_2)$ for any $\bm{x}_3$. Then, 
$$\mathbb{E}[ \{ \widehat{f}_{23}(\bm{z})^{-1} \widehat{f}(\bm{x}) - {f}_2(\bm{x}_2)^{-1} {f}_{12}(\bm{x}_1, \bm{x}_2) \}^2] = O(h_{11}^4 + h_{33}^{-4} + n^{-1} \|\bm{\Psi}\|^{-1}),
$$
where $f_{12}$ and $f_2$ are the probability density function of $(\bm{X}_{1,i},\bm{X}_{2,i})$ and $\bm{X}_{2,i}$ respectively.
\end{theorem}

\begin{proof}
Since the AMSE is asymptotically given by
$$\{k_3(\bm{0})\}^{-2} \mathbb{E}[ \{
\|\bm{H}_{33}\| (\widehat{f}(\bm{x})- \|\bm{H}_{33}\|^{-1} k_3(\bm{0}) f_{12}(\bm{x}_1, \bm{x}_2)) - \|\bm{H}_{33}\|^2 f_{12}(\bm{x}_1, \bm{x}_2) (\widehat{f}_3(\bm{x}_3) -k_3(\bm{0}))
\}^2],$$
applying Lemma 2 and the H\"{o}lder's inequality we have the result. 
\end{proof}

Theorem 4 shows the kernel conditional density estimator converges to the partially conditional density ${f}_2(\bm{x}_2)^{-1} {f}_{12}(\bm{x}_1, \bm{x}_2)$, where the rate of the convergence can depend on not $d= d_1 + d_2 + d_3$ but $d_1 + d_2$. 

\begin{cor}
Suppose the conditions of Theorem 4. If $h_{33}^{-1} = O(h_{11})$ or $o(h_{11})$, minimizing the AMSE yields $h_{11} = O(n^{-(d_1+d_2 +4)^{-1}})$ and the MSE with the optimal bandwidth is of order $n^{-4(d_1 +d_2 + 4)^{-1}}$. 
\end{cor}

\section{Case of multi-index model in regression or conditional density estimation}

Let us consider the case of multi-index model i.e. if there exist $d_2 \times (d_2 + d_3)$ matrix $\bm{A}$ such that ${\rm rank} \bm{A} = d_2$ and
$$\mathbb{P}[\bm{X}_{1,i} = \bm{x}_1, \bm{Z}_i = \bm{z}_i] = \mathbb{P}[\bm{X}_{1,i} = \bm{x}_1 | \bm{A}\bm{Z}_i = \bm{A}\bm{z}_i] \mathbb{P}[\bm{Z}_i = \bm{z}_i],$$
where $\bm{A} = (\bm{a}_1, \cdots, \bm{a}_{d_2})'$ but $\bm{A}$ is unknown. Then, $\bm{a}_1, \cdots, \bm{a}_{d_2} \in \mathbb{R}^{d_2+d_3}$ are linearly independent. We can choose any independent vectors instead belonging to the span of $\bm{a}_1, \cdots, \bm{a}_{d_2}$.

For any linearly independent $\bm{b}_1, \cdots, \bm{b}_{d_3} \in \mathbb{R}^{d_2+d_3}$ being also independent of $\bm{a}_1, \cdots, \bm{a}_{d_2}$, where $\bm{B} := (\bm{b}_1, \cdots, \bm{b}_{d_3})'$, it always holds that 
$$\mathbb{P}[\bm{X}_{1,i} = \bm{x}_1 | \bm{Z}_i = \bm{z}_i] \equiv \mathbb{P}[\bm{X}_{1,i} = \bm{x}_1 | \bm{AZ}_i = \bm{Az}_i, \bm{BZ}_i = \bm{Bz}_i].$$
Putting $\bm{AZ}_i=: \bm{Y}_{2,i}$, $\bm{BZ}_i =: \bm{Y}_{3,i}$ the case arrives at that of conditional independence 
$$\mathbb{P}[\bm{X}_{1,i} = \bm{x}_1 | \bm{Y}_{2,i} = \bm{y}_{2}, \bm{Y}_{3,i} = \bm{y}_{3}] = \mathbb{P}[\bm{X}_{1,i} = \bm{x}_1 | \bm{Y}_{2,i} = \bm{y}_{2}]$$
(see Section 2.3). 

Let the $d$-dimensional r.v. be $\bm{Y}_i := (\bm{Y}_{1,i}', \bm{Y}_{2,i}', \bm{Y}_{3,i}')'$ and $\bm{D} := (\bm{A}', \bm{B}')'$, where $\bm{Y}_{1,i} = \bm{X}_{1,i}$. The probability density function $g_{23}$ of the $(d_2+d_3)$-dimensional random variable $(\bm{Y}_{2,i}', \bm{Y}_{3,i}')' =: \bm{W}_i$ satisfies
$$g_{23}(\bm{w}) := \|\bm{D}\|^{-1} f_{23}(\bm{D}^{-1}\bm{w}) = \|\bm{D}\|^{-1} f_{23}(\bm{z}),$$
where $f_{23}$ is the probability density function of $(\bm{X}_{2,i},\bm{X}_{3,i})$ and $\bm{w} := \bm{D}\bm{z}$.

\subsection{Case of multi-index model in regression}

Let the kernel density estimator of $\bm{W}$ with the bandwidth matrix $\bm{\Omega}$ at the point $\bm{w} := \bm{D}\bm{z}$ be
\begin{align*}
\widehat{g}_{23}(\bm{w}) := (n \|\bm{\Omega}\|)^{-1} \sum_{i=1}^n k_{23}(\bm{\Omega}^{-1}(\bm{w}-\bm{W}_i)).
\end{align*}
Then, the followings hold that
\begin{align*}
\widehat{g}_{23}(\bm{w}) = (n \|\bm{D}^{-1}\bm{\Omega}\|)^{-1} \sum_{i=1}^n \|\bm{D}\|^{-1} k_{23}(\bm{\Omega}^{-1} \bm{D}(\bm{z}-\bm{Z}_i)) 
\end{align*}
can be seen as the estimator $\widehat{f}_{23}(\bm{z})$ with the bandwidth matrix $\bm{D}^{-1} \bm{\Omega}$ instead of $\bm{\Omega}$ and the kernel function $\|\bm{D}\|^{-1} k_{23}$, where $\bm{D}$ is regular. That means $\widehat{f}_{23}(\bm{z})$ actually takes the form $\widehat{g}_{23}(\bm{w})$, which is the density estimator of conditional dependence i.e. under the assumption of Lemma 2 it holds that 
\begin{align*}
\|\bm{D}^{-1} \bm{\Omega}\| \widehat{f}_{23}(\bm{z}) - k_{23}(\bm{0}) {g}_{2}(\bm{Az}) = O_P(n^{-1/2} h_{22}^{-d_2/2}) + O(h_{11}^2 + h_{22}^{-2}), 
\end{align*}
where $g_2$ is the probability density function of $\bm{Y}_{2,i}$.

Moreover,
\begin{align*}
\widehat{r}(\bm{w}) :=&  \{\widehat{g}_{23}(\bm{w})\}^{-1} (n \|\bm{\Omega}\|)^{-1} \sum_{i=1}^n \bm{Y}_{1,i} k_{23}(\bm{\Omega}^{-1}(\bm{w}-\bm{W}_i)) \\
=& \left\{ (n \|\bm{D}^{-1}\bm{\Omega}\|)^{-1} \sum_{i=1}^n k_{23}(\bm{\Omega}^{-1} \bm{D}(\bm{z}-\bm{Z}_i)) \right\}^{-1} (n \|\bm{D}^{-1} \bm{\Omega}\|)^{-1} \sum_{i=1}^n \bm{X}_{1,i} k_{23}(\bm{\Omega}^{-1} \bm{D}(\bm{z}-\bm{Z}_i))
\end{align*}
can be seen as the estimator $\widehat{m}(\bm{z})$ with the bandwidth matrix $\bm{D}^{-1} \bm{\Omega}$ instead of $\bm{\Omega}$ and the kernel function $k_{23}$. Then, under the assumption of Theorem 3 it holds that 
\begin{align*}
\mathbb{E}[\{\widehat{m}(\bm{z}) - r(\bm{Az})\}^2] = O(h_{22}^4 + h_{33}^{-4} + n^{-1} h_{22}^{-d_2}),
\end{align*}
where
$$r(\bm{Az}) := \int \bm{w}_{1} \frac{g_{12}(\bm{w}_1, \bm{Az})}{g_{2}(\bm{Az})} {\rm d}\bm{w}_1,$$
$g_{12}$ is the probability density function of $(\bm{Y}_{1,i},\bm{Y}_{2,i})$.

It follows that in case of multi-index model the optimal convergence rates of the regression estimator depends on not $(d_2+d_3)$ but $d_2$.

\begin{remark}
There exists $\bm{D}^{-1} \bm{\Omega}$ s.t. $\bm{\Omega}$ satisfies the corresponding assumption (e.g. {\rm (i)}--{\rm (iii)} in Theorem 3), which is seen by calculating backwards. The form of the matrix is
\begin{align*}
\bm{D}^{-1} \bm{\Omega} = \bm{D}^{-1} \left(\begin{matrix}
 h_{22} \bm{C}_{22} & h_{22} \bm{C}_{23} \\
 h_{22} \bm{C}_{32} & h_{33} \bm{C}_{33} \\
 \end{matrix}
 \right),
\end{align*}
where $\bm{C}_{22}$, $\bm{C}_{23}$, $\bm{C}_{32}$ and $\bm{C}_{33}$ are constant matrices with the size of $d_2 \times d_2$, $d_2 \times d_3$, $d_3 \times d_2$, and $d_3 \times d_3$ satisfying $\bm{C}_{22} \neq \bm{O}$ being regular and $\bm{C}_{33} \neq \bm{O}$. The order of the elements of the optimal $\bm{D}^{-1} \bm{\Omega}$ is the form of
$$\left(\begin{matrix}
h_{22}^{-1} & h_{33}^{-1} \\
h_{22}^{-1} & h_{33}^{-1} \\
 \end{matrix}
 \right).$$

The expectation of the kernel density estimator $\widehat{f}_{23}$ with the bandwidth matrix $\bm{D}^{-1}\bm{\Omega}$ can be seen to be $\|\bm{D}\|^{-1}$ times the probability density function of the random variable 
$$\bm{W} + \left(\begin{matrix}
 h_{22} \bm{C}_{22} & h_{22} \bm{C}_{23} \\
 h_{22} \bm{C}_{32} & h_{33} \bm{C}_{33} \\
 \end{matrix}
 \right) \bm{K}.$$
$\bm{K} := (\bm{K}_{2}', \bm{K}_{3}')'$ is the random variable independent of $\bm{W}$, whose probability density function is $k_{23}$. Due to the magnifying and shrinking matrix for large enough $n$ the random variable is dominated by $(\bm{Y}_{2}', h_{33}(\bm{C}_{33} \bm{K}_3)')'$, where the last $d_3$ variable asymptotically follows uniform on $\mathbb{R}^{d_3}$.
\end{remark}

\begin{remark}
Let the kernel $k_{23}$ be spherically symmetric i.e. there exists $\ell_{23}: \mathbb{R}^{d_2 + d_3} \to \mathbb{R}$ s.t. $k_{23}(\bm{w}) = \ell_{23}(\|\bm{w}\|_2)$. Then, for any orthogonal matrix $\bm{R}$ it holds that
$$\ell_{23}(\|\bm{w}\|_2) \equiv \ell_{23}(\|\bm{R} \bm{w}\|_2),$$
which means it is necessary for not $\bm{\Omega}$ but $\bm{\Omega} \bm{R}^{-1}$ to satisfy the corresponding assumption (e.g. {\rm (i)}--{\rm (iii)} in Theorem 3). The form of the matrix is
\begin{align*}
\bm{D}^{-1} \bm{\Omega} \bm{R}^{-1} = \bm{D}^{-1} \left(\begin{matrix}
 h_{22} \bm{C}_{22} & h_{22} \bm{C}_{23} \\
 h_{22} \bm{C}_{32} & h_{33} \bm{C}_{33} \\
 \end{matrix}
 \right) \bm{R}^{-1};
\end{align*}
however, this also cannot theoretically be block diagonal (let alone diagonal) for any orthogonal matrix $\bm{R}$.

The expectation of the kernel density estimator $\widehat{f}_{23}$ with the bandwidth matrix $\bm{D}^{-1} \bm{\Omega} \bm{R}^{-1}$ is $\|\bm{D}\|^{-1}$ times the probability density function of the random variable $\bm{W} + \bm{\Omega} \bm{R}^{-1}\bm{K}$. $\bm{K} := (\bm{K}_{2}', \bm{K}_{3}')'$ is the random variable independent of $\bm{W}$, whose probability density function is $\ell_{23}$. Hence $\bm{R}^{-1}\bm{K} = \bm{K}$ in distribution for spherically symmetric $k_{23}$. For large enough $n$ the random variable is dominated by $(\bm{Y}_{2}', (\bm{H}_{33} \bm{K}_3)')'$, where the last $d_3$ variable asymptotically follows uniform on $\mathbb{R}^{d_3}$.
\end{remark}

\subsection{Case of multi-index model in conditional density estimation}

Let us consider the same setting as Section 4.1. We also use the notation of the bandwidth matrix
\begin{align*}
 \left(\begin{matrix}
 \bm{\Psi} & \bm{H}_{12,3} \\
 \bm{H}_{3,12} & \bm{H}_{33} \\
 \end{matrix}
 \right) := \bm{H}
\end{align*}
 and consider the assumption of Lemma 2, where $(\bm{\Psi})_{ij}=O(h_{11})$, $(\bm{H}_{3,12})_{ij}=O(h_{33})$ and $(\bm{H}_{33})_{ij}=O(h_{33})$. $\bm{H}_{12,3}=\bm{O }$ or $(\bm{H}_{12,3})_{ij}=O(h_{33})$. Let $\widehat{g}$ be the kernel density estimator of $\bm{Y}_i$ at the point $\bm{y} := \widetilde{\bm{D}} \bm{x}$ given by
\begin{align*}
\widehat{g}(\bm{y}) := (n \|\bm{H}\|)^{-1} \sum_{i=1}^n k(\bm{H}^{-1}(\bm{y}-\bm{Y}_i)),
\end{align*}
where
$$ 
\widetilde{\bm{D}} := \left(\begin{matrix}
 \bm{E}_{11} & \bm{O} \\
 \bm{O} & \bm{D} \\
 \end{matrix}
 \right).
$$
Then, 
\begin{align*}
\widehat{g}(\bm{y}) = (n \|\widetilde{\bm{D}}^{-1} \bm{H}\|)^{-1} \sum_{i=1}^n \|\bm{D}\|^{-1} k(\bm{H}^{-1} \widetilde{\bm{D}}(\bm{x}-\bm{X}_i))
\end{align*}
can be seen as the estimator $\widehat{f}(\bm{x})$ with the bandwidth matrix $\widetilde{\bm{D}}^{-1} \bm{H}$ and the kernel function $\|\bm{D}\|^{-1} k$. Under the assumption of Theorem 4 it holds that 
\begin{align*}
\|\widetilde{\bm{D}}^{-1} \bm{H}\| \widehat{f}(\bm{x}) - k(\bm{0}) g(\widetilde{\bm{D}} \bm{x}) = O_P(n^{-1/2} h_{11}^{-(d_1 +d_2)/2}) + O(h_{11}^2 + h_{22}^{-2}),
\end{align*}
where $g$ denotes the probability density function of $\bm{Y}$. Since 
\begin{align*}
\widehat{g}_{23}(\bm{w}) := (n \|\bm{D}^{-1}\bm{\Omega}\|)^{-1} \sum_{i=1}^n \|\bm{D}\|^{-1} k_{23}(\bm{\Omega}^{-1} \bm{D}(\bm{z}-\bm{Z}_i))
\end{align*}
can be seen as the estimator $\widehat{f}_{23}(\bm{z})$ with the bandwidth matrix $\bm{D}^{-1} \bm{\Omega}$ and the kernel function $\|\bm{D}\|^{-1} k$, $\widehat{g}_{23}(\bm{D} \bm{z})^{-1} \widehat{g}(\widetilde{\bm{D}} \bm{x})$ can be seen as the estimator $\widehat{f}_{23}(\bm{z})^{-1} \widehat{f}(\bm{x})$ with the bandwidth matrix $\widetilde{\bm{D}}^{-1} \bm{H}$ and the kernel function $k$. If $\widetilde{\bm{D}}^{-1} \bm{H}$ satisfies the corresponding assumption in Theorem 4, all the previous results hold, where such $\bm{H}$ is easily obtained by calculating backwards. Here also note Remark 3.

Thus, we see 
$$\mathbb{E}[ \{ \widehat{f}_{23}(\bm{z})^{-1} \widehat{f}(\bm{x}) - {g}_{2}(\bm{Az})^{-1} g_{12}(\bm{x}_1, \bm{Az}) \}^2] = O(h_{11}^4 + h_{33}^{-4} + n^{-1} h_{11}^{-d_1-d_2}).
$$

It follows that in case of multi-index model the optimal convergence rates of the conditional density estimator depends on not $(d_2+d_3)$ but $d_2$.

\section{Simulation and Case study}

This section reports numerical properties of the kernel estimators of conditional density and regression. The bandwidth matrices $\bm{H}$ of the regression estimators were determined by several approaches: the leave-one-out least squares cross validation (LSCV) with an optimization algorithm or \texttt{npregbw} function in the \textsf{np} package in R (\citealt{li2004cross}) and MEKRO (\citealt{white2017variable}). The mean integrated squared errors (MISE) 
$$\mathbb{E}_{\bm{T}}[\mathbb{E}[\{\widehat{m}(\bm{X}) -m(\bm{X})\}^2|\bm{T}]$$
was also measured, where $\bm{T}$ denotes the training data following the distribution of $\bm{X}$. In kernel conditional density estimation the the bandwidth matrix was determined by such as the likelihood cross-validation and the least-squares cross-validation (\citealt{hall2004cross}). \texttt{npcdensbw} function in the \texttt{np} package calculated the bandwidth matrix. Every kernel function is the Gaussian one, which is spherically symmetric, hence all the bandwidth matrices are supposed to be symmetric (see Remark 1).

The following two settings are the case of conditional independence, which is given in \cite{white2017variable}. The first  case is with $d_1=1$, $d_2=2$, $d_3=1$ 
$$Y=\sin(2\pi X_1) + \sin(\pi X_2) + 0.5 \epsilon_0,$$
where $(X_1, X_2, X_3)$ is the trivariate uniform random variable on $[0,1]^3$. The mutually independent variable $\epsilon_0$ follows the standard normal distribution. 

The second case is with $d_1=1$, $d_2=2$, $d_3=8$
$$Y=\sin\{2\pi (X_1 + X_2)/(1+X_3)\} + \epsilon_1,$$
where $(X_1, \cdots, X_{10})$ is the uniform random variable on $[0,1]^{10}$ and $\epsilon_1$ follows the normal distribution with $\mathbb{V}[\epsilon]= \mathbb{V}[2\pi (X_1 + X_2)/(1+X_3)]/3$. The model is multi-index, where 
$$\bm{A} = \left(\begin{matrix}
 1 & 1 & 0 \\
 0 & 0 & 1 \\
 \end{matrix}
 \right)$$
 is an example satisfying $\mathbb{E}[\bm{X}_{1,i} | \bm{Z} = \bm{z} ] = m(\bm{A}\bm{z})$.

Tables 1--5 and 6--10 correspond to the first and the second case respectively. Tables 1--4 show the bandwidth values and the MISE of the regression estimator. The brute-force searched bandwidth matrix $\bm{H}$ was supposed to be 
$$\bm{H}_s := \left(\begin{matrix}
 h & 0 & 0 \\
 0 & h & 0 \\
 0 & 0 & h \\
 \end{matrix}
 \right), 
 \bm{H}_d := \left(\begin{matrix}
 h_1 & 0 & 0 \\
 0 & h_2 & 0 \\
 0 & 0 & h_3 \\
 \end{matrix}
 \right),
 \bm{H}_c := \left(\begin{matrix}
 h_1 & h_4 & h_5 \\
 h_4 & h_2 & h_6 \\
 h_5 & h_6 & h_3 \\
 \end{matrix}
 \right),
 $$
where the degree of freedom is $1,3, 6$. MEKRO and \texttt{npregbw} searched the optimal matrix taking the form of $\bm{H}_d$; however, the elements are supposed to be positive, which makes the difference from the brute-force searched bandwidth matrix. Table 5 shows the values of the bandwidth matrix 
$$\left(\begin{matrix}
 h_0 & 0 & 0 & 0\\
 0 & h_1 & 0 & 0\\
 0 & 0 & h_2 & 0\\
 0 & 0 & 0 & h_3\\
 \end{matrix}
 \right)$$
of the conditional density estimator chosen by \texttt{npcdensbw}.

The third case is
$$Y=2(X_1 + X_2) + X_3 + X_4 + \epsilon,$$
which comes from \cite{conn2019oracle}. $(X_1, \cdots, X_9)$ is the uniform random variable on $[0,1]^9$, and $\epsilon$ follows the standard normal distribution. The model is single-index, and the degrees are $d_1=1$, $d_2=1$ and $d_3=8$. Tables 11--16 show the bandwidth values and the mean integrated squared errors (MISE). 

\begin{table}[h]
\caption{MISE Results}
{\fontsize{8pt}{8pt}\selectfont
$$\begin{tabu}[c]{|c|cc|cc|cc|cc|cc|}
   \hline
n & {\rm scalar} & {\rm sd} & {\rm diagonal} & {\rm sd} & {\rm symmetric} & {\rm sd} & {\rm npregbw} & {\rm sd} & {\rm MEKRO} & {\rm sd} \\  \hline 
        \multicolumn{11}{|c|}{Y=\sin(2\pi X_1) + \sin(\pi X_2) + 0.5 \epsilon_0} \\ \hline
100 & 0.180 & 0.130 & 0.093 & 0.047 & 0.165 & 0.084 & 0.087 & 0.031 & 0.553 & 0.060 \\ \hline 
300 & 0.098 & 0.098 & 0.047 & 0.055 & 0.105 & 0.067 & 0.033 & 0.006 & 0.554 & 0.044 \\ \hline 
500 & 0.091 & 0.133 & 0.036 & 0.055 & 0.084 & 0.038 & 0.023 & 0.004 & 0.549 & 0.041 \\ \hline 
1000 & 0.050 & 0.076 & 0.015 & 0.053 & 0.048 & 0.050 & 0.014 & 0.003 & 0.370 & 0.257 \\ \hline 
 
 \multicolumn{11}{|c|}{Y=\sin\{2\pi (X_1 + X_2)/(1+X_3)\} + \epsilon_1} \\ \hline
100 & 0.439 & 0.063 & 0.324 & 0.178 & 0.498 & 0.212 & 0.210 & 0.092 & 0.495 & 0.037 \\ \hline 
300 & 0.372 & 0.047 & 0.139 & 0.129 & 0.152 & 0.166 & 0.079 & 0.010 & 0.501 & 0.022 \\ \hline 
500 & 0.344 & 0.057 & 0.101 & 0.113 & 0.093 & 0.117 & 0.053 & 0.006 & 0.492 & 0.018 \\ \hline 
1000 & 0.303 & 0.058 & 0.044 & 0.061 & 0.059 & 0.060 & 0.035 & 0.003 & 0.491 & 0.011 \\ \hline 
 
 \multicolumn{11}{|c|}{Y=2(X_1 + X_2) + X_3 + X_4 + \epsilon} \\ \hline
100 & 0.381 & 0.117 & 0.422 & 0.236 & 0.698 & 0.328 & 0.477 & 0.246 & 0.812 & 0.098 \\ \hline 
300 & 0.265 & 0.101 & 0.155 & 0.096 & 0.186 & 0.093 & 0.158 & 0.053 & 0.785 & 0.067 \\ \hline 
500 & 0.211 & 0.068 & 0.103 & 0.062 & 0.102 & 0.034 & 0.097 & 0.030 & 0.793 & 0.052 \\ \hline 
1000 & 0.176 & 0.081 & 0.066 & 0.046 & 0.035 & 0.018 & 0.059 & 0.017 & 0.788 & 0.053 \\ \hline 
 
  \end{tabu}$$
}
\end{table}

\section{Conclusion}\label{sec13}
Nonparametric estimators suffer from the curse of dimensionality; however, the kernel estimators are proved to depend only on the number of relevant variables to the dependent variables. 

This study does not propose bandwidth selection algorithms; however, existing methods such as the lease squares cross-validation allowing large bandwidth values are applicable. Suppose that the kernel function and all the elements of the Hessian matrix of the underlying density $f$ are square-integrable, which is usually assumed in the kernel estimation. Then, the convergence rate of the mean integrated squared error (MISE) is same as that of the AMSE. That means the corresponding elements of the bandwidth matrix minimizing MISE tends to infinity in the kernel estimation with irrelevant variables.

The restrictive assumption of the integrability of the moment of the underlying distribution can be relaxed by some data transformation such as log-transformation. The form of the transformed density estimator is
$$(n \|\bm{H}\|)^{-1} \sum_{i=1}^n (\prod_{j=1}^d x_j)^{-1} k(\bm{H}^{-1}(\ln\bm{x}- \ln\bm{X}_i)),$$
where $\ln$ is the elementwise transformation. The moment condition $\mathbb{E}[X_{i,k}^2] < \infty$ changes to $\mathbb{E}[(\ln X_{i,k})^2] < \infty$ without changing the convergence rate. Further investigation and its numerical properties are future work.

\bibliography{ref}

\end{document}